\pdfoutput=1
\documentclass{siamltex}

\usepackage{latexsym,amssymb,amsmath,enumerate,bbm,graphicx,psfrag,color}
\usepackage[latin1]{inputenc}
\usepackage{url}

\definecolor{mygray}{rgb}{0.3333,0.3333,0.3333}

\def\norm#1{\hspace{0.2ex} \|#1\| \hspace{0.2ex}} 

\newcommand{\labeq}[1]{\label{eq:#1}}
\def\req#1{{\rm(\ref{eq:#1})}}

\newcommand{\E}{\ensuremath{\mathbbm{E}}} 
\newcommand{\R}{\ensuremath{\mathbbm{R}}} 
\newcommand{\N}{\ensuremath{\mathbbm{N}}} 
\newcommand{\range}{\mathcal R}

\newcommand{\dx}[1][x]{\ensuremath{\,{\rm{d}} #1}}

\newcommand{\Ld}{L_\diamond}
\newcommand{\Hd}{H_\diamond}
\newcommand{\LL}{\mathcal L}

\newcommand{\Norm}[1]{\left\lVert#1\right\rVert}

\newcommand{\kommentar}[1]{}

\newtheorem{example}[theorem]{Example}
\newtheorem{remark}[theorem]{Remark}

\begin{document}

\title{Detecting stochastic inclusions in electrical impedance tomography}
       
\author{Andrea Barth\footnotemark[2], Bastian Harrach\footnotemark[3], Nuutti Hyv\"onen\footnotemark[4], and Lauri Mustonen\footnotemark[4]}
\renewcommand{\thefootnote}{\fnsymbol{footnote}}

\footnotetext[2]{Department of Mathematics, University of Stuttgart, Stuttgart, Germany (andrea.barth@math.uni-stuttgart.de).}
\footnotetext[3]{Institute of Mathematics, Goethe University Frankfurt, Frankfurt, Germany (harrach@math.uni-frankfurt.de). The work of BH was supported by the German Research Foundation (DFG) within the Cluster of Excellence in Simulation Technology (EXC 310/1) at the University of Stuttgart.}
\footnotetext[4]{Aalto University, Department of Mathematics and Systems Analysis, P.O. Box 11100, FI-00076 Aalto, Finland (nuutti.hyvonen@aalto.fi, lauri.mustonen@aalto.fi). The work of NH and LM was supported by the Academy of Finland (decision 267789).}
 
\maketitle

\pagestyle{myheadings}
\thispagestyle{plain}
\markboth{A. BARTH, B. HARRACH, N. HYV\"ONEN, AND L. MUSTONEN}{DETECTING STOCHASTIC INCLUSIONS IN EIT}

\begin{abstract}
This work considers the inclusion detection problem of electrical impedance tomography with stochastic conductivities. It is shown that a conductivity anomaly with a random conductivity can be identified by applying the Factorization Method or the Monotonicity Method to the mean value of the corresponding Neumann-to-Dirichlet map provided that the anomaly has high enough contrast in the sense of expectation. The theoretical results are complemented by numerical examples in two spatial dimensions.
\end{abstract}



%%%%%%%%%%%%%%%%%%%%%%%%%%%%%%%%%%%%%%%%%%%%%%%%%%%%%%%%%%%%%%%%%%%%%%%%%%%%
\section{Introduction}
\label{Sec:intro}
%%%%%%%%%%%%%%%%%%%%%%%%%%%%%%%%%%%%%%%%%%%%%%%%%%%%%%%%%%%%%%%%%%%%%%%%%%%%

{\em Electrical impedance tomography} (EIT) is an imaging modality for finding information about the internal conductivity of an examined physical body from boundary measurements of current and potential. The reconstruction task of EIT is nonlinear and highly illposed. EIT has potential applications in,~e.g.,~medical imaging, geophysics, nondestructive testing of materials and monitoring of industrial processes. For general information on EIT,
we refer to \cite{adler2011electrical,bayford2006bioimpedance,Borcea02,borcea2003addendum,Cheney99,holder2004electrical,lionheart2003eit,metherall1996three,Uhlmann09} and the references therein.

This work considers detection of {\em stochastic inclusions} by EIT. The conductivity inside the examined object is assumed to consist of a known deterministic background with embedded inclusions within which the conductivity is a random field. The inclusion shapes are deterministic. The available measurement is assumed to be the expected value of the corresponding Neumann-to-Dirichlet boundary operator, i.e., the expectation of the idealized current-to-voltage map.  We demonstrate that the Factorization Method and the Monotonicity Method can be utilized to reconstruct the supports of the stochastic inhomogeneities essentially in the same way as in the deterministic case if the inhomogeneities exhibit high enough conductivity contrast in the sense of expectation. Detecting inclusions with random conductivities from expected boundary measurements can be regarded an interesting problem setting from a mere mathematical standpoint, but it also has potential applications: consider,~e.g.,~temporally averaged boundary measurements corresponding to an inhomogeneity whose conductivity varies in time (e.g.,~a cancerous lesion subject to blood flow).

To the best of our knowledge, this is the first work that considers reconstructing properties of a random conductivity from the expectation of the boundary measurements. With that being said, it should be mentioned that for inverse scattering in the half-plane it has been shown that a single realization of the backscatter data at all frequencies uniquely defines the principal symbol of the covariance operator for a random scattering potential or a Robin parameter in the boundary condition~\cite{Helin14,Lassas08}. Moreover, time reversal related inverse scattering techniques have been used extensively for the detection of inclusions embedded in random media; see,~e.g.,~\cite{Bal05, Bal03, Borcea02b} and the references therein.

The Factorization Method was originally introduced  in the framework of inverse obstacle scattering by Kirsch \cite{Kirsch98}, and it was subsequently carried over to EIT by Br\"uhl and Hanke \cite{Bruhl01,Bruhl00}. The Factorization Method has attracted a considerable amount of attention in the scientific community over the past twenty years; in the context of EIT we refer the reader to \cite{chaulet2014factorization,choi2014regularizing,gebauer2006factorization,gebauer2008localized,Gebauer07,hakula2009computation,hanke2003recent,hanke2008factorization,harrach2009detecting,Harrach10,hyvonen2004complete,hyvonen2007numerical,kirsch2005factorization,Lechleiter08,nachman2007imaging,schmitt2009factorization,schmitt2011factorization}
and the overview articles \cite{Hanke11,harrach2013recent,Kirsch07} for more information on the developments since the seminal article of Kirsch. The Monotonicity Method \cite{harrach2013monotonicity} is a more recent inclusion detection technique that can be considered natural for EIT as it takes explicitly advantage of the coercivity of the underlying conductivity equation. In particular, unlike the Factorization Method, a variant of the Monotonicity Method can handle indefinite inclusions,~i.e.,~it stays functional even if some parts of the inhomogeneities are more and some other parts less conductive than the background. Monotonicity-based methods were first formulated and numerically tested by Tamburrino and Rubinacci \cite{Tamburrino06,Tamburrino02}. Their rigorous justification
is given in \cite{harrach2013monotonicity} based on the theory of localized potentials~\cite{gebauer2008localized}.

For simplicity, we employ the idealized {\em continuum model} for EIT, that is, we unrealistically assume to be able to drive any square-integrable current density into the object of interest and to measure the resulting potential everywhere on the boundary. Consult \cite{Hyvonen09} for analysis of the relationship between the continuum model and the {\em complete electrode model}, which is the most accurate model for real-world EIT measurements~\cite{Cheng89,Somersalo92}. Observe also that the Factorization Method has, in fact, been demonstrated to function also for realistic electrode measurements \cite{chaulet2014factorization,Harrach10,Lechleiter08}, and that the monotonicity method allows to characterize the achievable resolution in realistic measurement settings \cite{harrach2015resolution}.

This text is organized as follows. In Section~\ref{Sec:setting}, we introduce the stochastic Neumann-to-Dirichlet boundary map corresponding to a random conductivity field and examine how the fundamental monotonicity results for the related quadratic forms are carried over from the deterministic case. Subsequently, Section~\ref{Sec:detection1} formulates and proves the theorems that demonstrate the functionality of the Factorization and Monotonicity Methods for detecting stochastic inclusions. The two-dimensional numerical examples for the Factorization Method are presented in Section~\ref{Sec:numerics}; the stochastic measurement data are simulated by applying the stochastic finite element method (cf.,~e.g.,~\cite{Schwab11}) to the conductivity equation with a random coefficient.

%%%%%%%%%%%%%%%%%%%%%%%%%%%%%%%%%%%%%%%%%%%%%%%%%%%%%%%%%%%%%%%%%%%%%%%%%%%%
\section{The setting and monotonicity results}

\label{Sec:setting}
%%%%%%%%%%%%%%%%%%%%%%%%%%%%%%%%%%%%%%%%%%%%%%%%%%%%%%%%%%%%%%%%%%%%%%%%%%%%

In this section, we first review certain monotonicity properties of the deterministic Neumann-to-Dirichlet map associated with the conductivity equation. These inequalities are subsequently interpreted from the viewpoint of stochastic conductivities.

\subsection{The deterministic Neumann-to-Dirichlet operator}\label{subsect:NtD_det}

Let $D\subset\R^n$, $n\geq 2$ be a bounded domain with a smooth boundary $\partial D$ and outer unit normal vector~$\nu$. The subset of $L^\infty(D)$-functions with positive essential infima is denoted by $L_+^\infty(D)$. Moreover, $\Hd^1(D)$ and $\Ld^2(\partial D)$ are defined to be the subspaces of $H^1(D)$- and $L^2(\partial D)$-functions, respectively, with vanishing integral means over $\partial D$, i.e.,
\begin{align*}
\Hd^1(D)&= \left\{ u\in H^1(D):\ \int_{\partial D} u \dx[s]=0 \right\} ,\\[2mm]
\Ld^2(\partial D)&= \left\{ u\in L^2(\partial D):\ \int_{\partial D} u \dx[s]=0 \right\}.
\end{align*}

In the context of EIT, $D$ describes the imaging subject. A deterministic conductivity inside the subject is modeled by a function 
\[
\sigma:\ D\to \R 
\]
such that $\sigma\in L_+^\infty(D)$. The deterministic nonlinear forward operator of EIT is defined by
\[
L_+^\infty(D) \ni \sigma \mapsto \Lambda(\sigma) \in \LL(\Ld^2(\partial D)),
\]
where 
\[
\Lambda(\sigma): \ g\mapsto u_\sigma^g|_{\partial D}
\]
and $u_\sigma^g\in \Hd^1(D)$ solves
\begin{align}\labeq{EIT}
\nabla\cdot  \sigma \nabla u_\sigma^g  = 0 \quad {\rm in} \ D, \qquad \sigma \partial_\nu u_\sigma^g = g \quad {\rm on} \ \partial D .
\end{align}
The bilinear form associated with $\Lambda(\sigma)$ is
\begin{equation}\labeq{bilinear_form}
\int_{\partial D} g \Lambda(\sigma) h \dx[s] 
= \int_D \sigma \nabla u_\sigma^g \cdot \nabla u_\sigma^h \dx.
\end{equation}
It is well-known that $\Lambda(\, \cdot \,)$ is continuous and for each $\sigma\in L^\infty_+(D)$, the linear operator $\Lambda(\sigma):\Ld^2(\partial D) \to \Ld^2(\partial D)$ is self-adjoint and compact.

A useful tool in studying $\Lambda$ is the following monotonicity result
from Ikehata, Kang, Seo, and Sheen \cite{ikehata1998size,kang1997inverse}. It has frequently been used in deterministic inclusion detection, cf.,~e.g.,~\cite{harrach2013recent,harrach2009detecting,harrach2010exact,harrach2013monotonicity,ide2007probing,kirsch2005factorization}.

\begin{lemma}\label{lemma:IKSS}
Let $\sigma_0,\sigma_1\in L^\infty_+(D)$. Then, for all $g\in \Ld^2(\partial D)$,
\begin{align*}
 \int_D(\sigma_1-\sigma_0) |\nabla u_0|^2 \dx 
&\geq \int_{\partial D} g \left(\Lambda(\sigma_0)-\Lambda(\sigma_1)\right)g \dx[s]\\
&\geq \int_D \sigma_0^2 \left( \frac{1}{\sigma_0}-\frac{1}{\sigma_1} \right)|\nabla u_0|^2\dx,
\end{align*}
with $u_0:=u^g_{\sigma_0}\in H^1_\diamond(D)$ solving \req{EIT}.
\end{lemma}
\begin{proof}
A short proof can be found, e.g., in \cite[Lemma~3.1]{harrach2013monotonicity}.
\end{proof}

For two bounded self-adjoint operators
$A,B:\ \Ld^2(\partial D)\to \Ld^2(\partial D)$ we write
\[
A\geq B
\]
if the inequality holds in the sense of quadratic forms, i.e., if
\[
\int_{\partial D} g \left( A - B \right)g \dx[s]\geq 0\quad \mbox{ for all } g\in L^2_\diamond(\partial D).
\]
With this convention, Lemma~\ref{lemma:IKSS} yields that for all $\sigma_0,\sigma_1\in L^\infty_+(D)$,
\begin{align}\labeq{mono_direct}
\sigma_0\leq \sigma_1 \quad \text{ implies } \quad \Lambda(\sigma_1)\leq \Lambda(\sigma_0),
\end{align}
which is why we call Lemma~\ref{lemma:IKSS} a \emph{monotonicity estimate}.

\subsection{The stochastic Neumann-to-Dirichlet operator}\label{subsect:NtD_stoch}

We model stochastic objects by treating the conductivity as a random field
\[
\sigma:\ \Omega\to L^\infty_+(D),
\]
where $(\Omega,\Sigma,P)$ is a probability space. By pointwise concatenation
\[
\Lambda(\sigma)(\omega):=\Lambda(\sigma(\omega))
\]
the Neumann-to-Dirichlet operator becomes random as well.

To be able to rigorously introduce moments (such as expectation or variance), we recall the definition of the \emph{Lebesgue--Bochner space} $L^p(\Omega;X)$:
\begin{definition}
Given a Banach space $X$, a function 
\[
f:\ \Omega\to X
\]
is called \emph{strongly measurable} if it is measurable (with respect to $\Sigma$ and the Borel algebra on $X$) and essentially separably valued. The space $L^p(\Omega;X)$, $1\leq p< \infty$ 
is the set of (equivalence classes with respect to almost everywhere equality of)
strongly measurable functions
satisfying
\[
\norm{f}_{L^p(\Omega;X)}
:=
\left( \int_{\Omega} \| f(\omega) \|_{X}^{p} \dx[P] \right)^{1/p} < \infty.
\] 
For $Y\subseteq X$, $L^p(\Omega;Y)$ is defined as the subset of those $f\in L^p(\Omega;X)$ for which $f(\omega)\in Y$ almost surely.
\end{definition}

The following lemma shows that the expected value of the Neumann-to-Dirichlet operator $\Lambda(\sigma)$ is well defined if both the conductivity $\sigma$ and its reciprocal, the resistivity $\sigma^{-1}$, have well defined expectations.
\begin{lemma}
Let $\sigma,\sigma^{-1}\in L^1(\Omega;L^\infty_+(D)) \subseteq L^1(\Omega;L^\infty(D))$. 

Then 
$
\Lambda(\sigma)\in L^1(\Omega;\LL(\Ld^2(\partial D))),
$
and the following expectations are well defined:
\begin{align*}
\E(\sigma)&:=\int_\Omega \sigma(\omega) \dx[P]\in L^\infty_+(D),\\
\E(\sigma^{-1})&:=\int_\Omega \sigma^{-1}(\omega) \dx[P]\in L^\infty_+(D),\\
\E(\Lambda(\sigma))&:=\int_\Omega \Lambda(\sigma)(\omega) \dx[P]\in \LL(\Ld^2(\partial D)).
\end{align*}
Furthermore, $\E(\Lambda(\sigma)): \Ld^2(\partial D) \to \Ld^2(\partial D)$ is self-adjoint and compact.
\end{lemma}

\begin{proof}
The fact that $\E(\sigma)$ and $\E(\sigma^{-1})$ are well defined and belong to  $L^\infty_+(D)$ follows directly from the assumption $\sigma,\sigma^{-1}\in L^1(\Omega;L^\infty_+(D))$. Furthermore,  $\Lambda(\sigma):\ \Omega\to \LL(\Ld^2(\partial D))$ 
is measurable since it is a concatenation of a measurable and a continuous function. 
Since continuous images of separable sets are separable, the same argument shows that $\Lambda(\sigma)$ is also strongly measurable. 

To show that $\Lambda(\sigma)$ has a finite expectation, we fix
the constant deterministic conductivity $\sigma_0 \equiv 1\in L^\infty_+(\Omega)$, 
and employ the monotonicity lemma \ref{lemma:IKSS} and \req{bilinear_form}. It follows that for almost all $\omega\in \Omega$,
\begin{align*}
\frac{\norm{\Lambda(\sigma(\omega))}_{\LL(\Ld^2(\partial D))}}{\norm{\Lambda(1)}_{\LL(\Ld^2(\partial D))}}
& \leq 
1+ \frac{\norm{\Lambda(\sigma(\omega))-\Lambda(1)}_{\LL(\Ld^2(\partial D))}}{\norm{\Lambda(1)}_{\LL(\Ld^2(\partial D))}}\\[1mm]
& = 
1+ \frac{\sup_{\norm{g}_{\Ld^2(\partial D)}=1} \left| \int_{\partial D} g \left(\Lambda(\sigma) - \Lambda(1)\right) g \dx[s] \right|}%
{\sup_{\norm{g}_{\Ld^2(\partial D)}=1} \left| \int_{\partial D} g  \Lambda(1) g \dx[s] \right|} \\[1mm]
& \leq  1+ \max\left\{ \Norm{1-\sigma(\omega)}_{L^\infty(D)}, 
\Norm{ 1- \sigma(\omega)^{-1}}_{L^\infty(D)}
\right\}\\[1mm]
&\leq 2+\max \left\{ \Norm{\sigma(\omega)}_{L^\infty(D)},
\Norm{\sigma(\omega)^{-1}}_{L^\infty(D)}
\right\}.
\end{align*}
Since $\sigma,\sigma^{-1}\in L^1(\Omega;L^\infty_+(D))$, this implies that $\Lambda(\sigma)\in L^1(\Omega;\LL(\Ld^2(\partial D)))$. 

The self-adjointness and compactness of $\E(\Lambda(\sigma))$ follow as in the deterministic case.
\end{proof}

Our stochastic framework covers an anomaly of deterministic size with a random conductivity value but not an anomaly with random size. We will demonstrate this in the following example.
\begin{example}\label{ex:example1}
We consider a domain $D$ with homogeneous unit background conductivity and an embedded anomaly.
\begin{enumerate}[(a)]
\item First assume that the anomaly consists of finitely many, mutually disjoint pixels, where the conductivity has a constant, uniformly distributed, stochastic value. To be more precise,
\[
\sigma(\omega)= 1 + \sum_{q=1}^Q \theta_q(\omega) \chi_q, \qquad \omega\in \Omega,
\]
where $\chi_q$ is the characteristic function of the $q$th pixel
and 
\[
\theta_q:  \Omega \to [a_q,b_q] \subset (-1, \infty), \qquad q = 1, \dots, Q, 
\]
are mutually independent, uniformly distributed random variables. Hence, each $\theta_q$ has a probability density function given by $(b_q-a_q)^{-1}\chi_{[a_q,b_q]}$.
We then have, using the linearity of the expectation operator and the disjointness of the pixels
\begin{align*}
\E(\sigma)&= 1 + \sum_{q=1}^Q \E(\theta_q) \chi_q = 1 + \sum_{q=1}^Q \int_\Omega\theta_q(\omega)\,{\rm d}P \;\chi_q\\
& = 1 + \sum_{q=1}^Q \frac1{b_q-a_q}\int_{a_q}^{b_q}x\,{\rm d} x \;\chi_q = 1 + \sum_{q=1}^Q \frac{a_q+b_q}{2} \chi_q,\\[1mm]
\E(\sigma^{-1})& = 
\left\{ \begin{array}{l l}
\E\left( \frac{1}{1+\theta_q \chi_q} \right) = (b_q-a_q)^{-1}\log\frac{1+b_q}{1+a_q} & \text{ in the $q$th pixel,}\\[2mm]
1 & \text{ in the background.}
\end{array}
\right.
\end{align*}
\item Now assume that the domain $D=B_1(0)\subseteq \R^2$ is the unit disk
that contains a circular anomaly with stochastic radius, say,
\[
\sigma(\omega)=1 + \chi_{\theta(\omega)}, \qquad \omega\in \Omega,
\]
where $\theta: \Omega \to [0,1/2]$ is uniformly distributed and $\chi_{\theta}$ is the characteristic function of a disk with radius $\theta$ centered at the origin. Then for $\omega_1, \omega_2 \in \Omega$ such that $\theta(\omega_1) \not = \theta(\omega_2)$, it holds that
\[
\norm{\sigma(\omega_1)- \sigma(\omega_2)}_{L^\infty(D)} = \norm{\chi_{\theta(\omega_1)}-\chi_{\theta(\omega_2)}}_{L^\infty(D)}=1,
\]
which shows that $\sigma:\ \Omega \to L^\infty_+(D)$ cannot be essentially separably valued.
\end{enumerate}
\end{example}

By replacing $\sigma_1$ in Lemma \ref{lemma:IKSS}  by a random conductivity $\sigma$ and taking expectations, we immediately obtain the following stochastic monotonicity estimate.
\begin{corollary}\label{cor:monotonicity}
Consider a deterministic conductivity $\sigma_0$ and a stochastic conductivity $\sigma$, i.e., 
$\sigma_0\in L^\infty_+(D)$ and $\sigma,\sigma^{-1}\in L^1(\Omega;L^\infty_+(D))$. Then, for all $g\in \Ld^2(\partial D)$,
\begin{align*}
\int_D(\E(\sigma)-\sigma_0)) |\nabla u_0|^2 \dx
 & \geq \int_{\partial D} g \big(\Lambda(\sigma_0)-\E(\Lambda(\sigma))\big) g \dx[s]\\
& \geq \int_D \sigma_0^2 \left( \sigma_0^{-1}-\E(\sigma^{-1}) \right)|\nabla u_0|^2\dx,
\end{align*}
where $u_0:=u^g_{\sigma_0}\in H^1_\diamond(D)$ solves \req{EIT}.
In particular, by choosing in turns $\sigma_0 =\E(\sigma)$ and $\sigma_0 =\E(\sigma^{-1})^{-1}$, it follows  that
\begin{align}\labeq{mono_direct_E}
\Lambda(\E(\sigma))\leq \E(\Lambda(\sigma)))\leq \Lambda \big(\E(\sigma^{-1})^{-1}\big).
\end{align}
\end{corollary}

%%%%%%%%%%%%%%%%%%%%%%%%%%%%%%%%%%%%%%%%%%%%%%%%%%%%%%%%%%%%%%%%%%%%%%%%%%%%
\section{Detecting stochastic inclusions}
\label{Sec:detection1}
%%%%%%%%%%%%%%%%%%%%%%%%%%%%%%%%%%%%%%%%%%%%%%%%%%%%%%%%%%%%%%%%%%%%%%%%%%%%

We consider a domain $D$ that has known deterministic background properties but contains an anomaly $A$ with random conductivity. For simplicity, we derive our results only for the case of a homogeneous unit background conductivity and assume that
$A$ is an open set for which $\overline A\subset D$ has a connected complement. Following the line of reasoning in \cite{harrach2013recent,harrach2013monotonicity}, it should be straightforward to 
generalize the results for (known and deterministic) piecewise analytic background conductivities and a more general subset $A$. In the rest of this section, $u_0:=u^g_{\sigma_0}\in H^1_\diamond(D)$ solves \req{EIT} for the considered background conductivity $\sigma_0 \equiv 1$ and a current density $g \in L^2_{\diamond}(\partial D)$ that should be clear from the context.

In the deterministic case, rigorous noniterative inclusion detection methods such as the Factorization Method or the Monotonicity Method
are known to correctly recover the shape of an unknown anomaly whenever there exists $\alpha>0$ such that the conductivity in the anomaly is either
\[
\text{(a) larger than $1+\alpha$,} \quad \text{ or } \quad \text{(b) smaller than $1-\alpha$.}
\]
As noted above, here we consider a domain $D$ containing a stochastic anomaly $A$, that is, 
\[
\sigma(\omega,x):=\left\{ \begin{array}{l l } 1 & \quad \text{ for } x\in D\setminus A,\\
\sigma_A(\omega,x) & \quad \text{ for } x\in A,\end{array}\right.
\]
where $\sigma_A:\ \Omega\to L^\infty_+(A)$ is a random field with $\sigma_A,\sigma_A^{-1}\in L^1(\Omega;L^\infty_+(A))$. It is tempting to try to detect $A$ from the stochastic boundary measurements by applying the deterministic Factorization or Monotonicity Method to the expected value $\E(\Lambda(\sigma))$ of the stochastic Neumann-to-Dirichlet boundary operator. We will show that these methods indeed correctly recover the anomaly $A$ if there exists $\alpha>0$ such that, almost everywhere in $A$, either
\begin{equation}
\label{stoch_cond}
\text{(a) $\E(\sigma_A^{-1})^{-1}  \geq 1+ \alpha$,} \quad \text{ or } \quad \text{(b) $\E(\sigma_A)  \leq 1 -\alpha$.}
\end{equation}
Note that, by Jensen's inequality, $\E(\sigma_A)\geq \E(\sigma_A^{-1})^{-1}$, meaning that (a) also implies that $\E(\sigma_A) \geq 1+ \alpha$. Analogously, (b) implies that $\E(\sigma_A^{-1})^{-1}\leq 1-\alpha$.

We start with the Factorization Method. 

\begin{theorem}[Factorization Method]
\label{thm:factorization}
Assume that there exists $\alpha>0$ such that one of the two conditions in
\eqref{stoch_cond} is valid.
Then for any nonvanishing dipole moment $d\in \R^n$ and every point $z\in D\setminus \partial A$, it holds that
\[
z\in A\quad \text{ if and only if } \quad \Phi_{z,d}|_{\partial D}\in \range \left( \left| \E(\Lambda(\sigma))- \Lambda(1) \right|^{1/2} \right),
\]
where $\Phi_{z,d} \in C^{\infty}(\overline{D} \setminus \{z\}) \cap H^{-n/2+1-\varepsilon}(D)$, $\varepsilon > 0$, is the unique solution of 
\[
\Delta \Phi_{z,d}=d\cdot \nabla \delta_z \quad \text{in } D, \qquad \partial_\nu \Phi_{z,d} = 0 \quad \text{on } \partial D, \qquad \int_{\partial D} \Phi_{z,d}\dx[s]  = 0,
\]
i.e., the potential created by a dipole at $z$ corresponding to the moment $d$, the unit conductivity and vanishing current density on $\partial D$.
\end{theorem}
\begin{proof}
In our current setting, the monotonicity result of Corollary~\ref{cor:monotonicity} reads
\begin{align*}
\int_A(\E(\sigma_A)-1)) |\nabla u_0|^2 \dx
 & \geq \int_{\partial D} g \big(\Lambda(1)-\E(\Lambda(\sigma))\big) g \dx[s]\\
 & \geq \int_A  \left( 1-\E(\sigma_A^{-1}) \right)|\nabla u_0|^2\dx
\end{align*}
for all $g\in \Ld^2(\partial D)$. 
In case (a), we have $\E(\sigma_A^{-1})\leq (1+\alpha)^{-1}$, and thus for all $g\in \Ld^2(\partial D)$,
\begin{align*}
\| \E(\sigma_A) \|_{L^\infty(A)} \int_A |\nabla u_0|^2 \dx
 & \geq \int_{\partial D} g \big(\Lambda(1)-\E(\Lambda(\sigma))\big) g \dx[s]\\
 & \geq  \frac{\alpha}{1+\alpha} \int_A  |\nabla u_0|^2\dx.
\end{align*}
In case (b), we use $\E(\sigma_A) \leq 1- \alpha$ to obtain
\begin{align*}
\alpha \int_A |\nabla u_0|^2 \dx
 & \leq \int_{\partial D} g \big(\E(\Lambda(\sigma)) -\Lambda(1)\big) g \dx[s]
 \leq  \big\| \E(\sigma_A^{-1}) \big \|_{L^\infty(A)} \int_A |\nabla u_0|^2 \dx
\end{align*}
for all $g\in \Ld^2(\partial D)$.
The assertion now follows from the previous two estimates by the same arguments as in the proof of the deterministic Factorization Method in \cite[Theorem~5]{harrach2013recent}.
\end{proof}

Then it is the turn of the Monotonicity Method.

\begin{theorem}[Monotonicity Method]\label{thm:MonotonicityMethod}
\begin{enumerate}[(a)]
\item For every constant $\alpha>0$ such that $\E(\sigma_A^{-1})^{-1} \geq 1+\alpha$ almost everywhere on $A$, and for every open ball $B\subset D$,
\[
B\subseteq A \quad \text{ if and only if } \quad \Lambda(1+\alpha \chi_B)\geq \E(\Lambda(\sigma_A)).
\]
\item For every $0 < \alpha < 1$ such that $\E(\sigma_A) \leq 1-\alpha$ almost everywhere on $A$, and for every open ball $B\subset D$,
\[
B\subseteq A \quad \text{ if and only if } \quad \Lambda(1-\alpha \chi_B)\leq \E(\Lambda(\sigma_A)).
\]
\end{enumerate}
\end{theorem}
\begin{proof}
We prove the two cases separately.
\begin{enumerate}[(a)]
\item 
If $B\subseteq A$, then $1+\alpha \chi_B\leq \E(\sigma^{-1})^{-1}$, and we obtain from
\req{mono_direct} and \req{mono_direct_E} that
\[
\Lambda(1+\alpha \chi_B) \geq \Lambda(\E(\sigma^{-1})^{-1}) \geq \E(\Lambda(\sigma)).
\]

Now, let $B\not\subseteq A$. We will show that $\Lambda(1+\alpha \chi_B)\not\geq \E(\Lambda(\sigma))$.
Since \req{mono_direct} proves that shrinking $B$ enlarges $\Lambda(1+ \alpha \chi_B)$,
we may assume without loss of generality  that $B\subseteq D\setminus \overline A$. Using the
monotonicity result in Corollary~\ref{cor:monotonicity}, we obtain that
\begin{align*}
\lefteqn{\int_D g \big(\Lambda(1+\alpha \chi_B) - \E(\Lambda(\sigma))\big) g \dx}\\
& \leq \int_D \big( \E(\sigma)- (1+\alpha \chi_B)\big) |\nabla u_0|^2\dx\\
&  = \int_A \left( \E(\sigma_A) - 1 \right) |\nabla u_0|^2 \dx - \alpha \int_B |\nabla u_0|^2 \dx
\end{align*}
for all $g\in \Ld^2(\partial D)$.
Using a localized potential with large energy on $B$ and small energy on $A$ 
(see \cite{gebauer2008localized, harrach2013monotonicity}), it
follows that there exists $g\in \Ld^2(\partial D)$ satisfying
\[
\int_D g \big(\Lambda(1+\alpha \chi_B) - \E(\Lambda(\sigma))\big) g \dx < 0,
\]
so that $\Lambda(1+\alpha \chi_B) - \E(\Lambda(\sigma))\not\geq 0$.

\item If $B\subseteq A$, then $1-\alpha \chi_B\geq  \E(\sigma)$, and we obtain from
\req{mono_direct} and \req{mono_direct_E} that
\[
\Lambda(1-\alpha \chi_B) \leq \Lambda(\E(\sigma)) \leq \E(\Lambda(\sigma)).
\]

Now, let $B\not\subseteq A$. To show that $\Lambda(1-\alpha \chi_B)\not\leq \E(\Lambda(\sigma))$,
we can again assume without loss of generality that $B\subseteq D\setminus \overline A$. Employing Corollary~\ref{cor:monotonicity} with $\sigma_0 = 1+\alpha \chi_B$, we obtain this time around that 
\begin{align*}
\int_D g \big(\Lambda(1 - \alpha \chi_B) &- \E(\Lambda(\sigma))\big) g \dx \\[1mm]
& \geq \int_D (1-\alpha\chi_B)^2 \left( (1-\alpha\chi_B)^{-1} - \E(\sigma^{-1}) \right) |\nabla u_0|^2\dx\\[1mm]
& = \int_A   \left( 1 - \E(\sigma_A^{-1}) \right) |\nabla u_0|^2\dx  
+  \alpha(1-\alpha) \int_B  |\nabla u_0|^2\dx
\end{align*}
for all $g\in \Ld^2(\partial D)$.
As in case (a), we can use a localized potential with large energy on $B$ and small energy on $A$ 
to conclude that there exists $g\in \Ld^2(\partial D)$ satisfying
\[
\int_D g \big(\Lambda(1-\alpha \chi_B) - \E(\Lambda(\sigma))\big) g \dx > 0,
\]
so that $\Lambda(1-\alpha \chi_B) - \E(\Lambda(\sigma))\not\leq 0$.
\end{enumerate}
This completes the proof.
\end{proof}

Theorem~\ref{thm:MonotonicityMethod} provides an explicit characterization of the anomaly: If a suitable threshold value $\alpha > 0$ is available, one can test whether any given open ball is enclosed by $A$ (assuming infinite measurement/numerical precision).
The following theorem is essential in the numerical implementation of the Monotonicity Method. It demonstrates that, as in the deterministic case, the monotonicity test can be performed using the linearized Neumann-to-Dirichlet map, which considerably reduces the computational cost as one does not need to simulate a new nonlinear Neumann-to-Dirichlet operator for testing each new ball and threshold value.
To get set, we define $\Lambda'(1)\chi_B\in \LL(\Ld^2(\partial D))$ for a ball $B\subseteq D$ via its quadratic form
\[
\int_{\partial D} g \, \Lambda'(1)\chi_B g \dx[s]:=- \int_B |\nabla u_0|^2 \dx \qquad {\rm for} \ {\rm all} \ g \in L^2_\diamond(\partial D),
\]
where once again $u_0 \in H^1_\diamond(D)$ is the reference potential corresponding to the unit background conductivity and the current density $g$.

\begin{theorem}[Linearized Monotonicity Method]
\begin{enumerate}[(a)]
\item If there exists $\alpha>0$ such that  $\E(\sigma_A^{-1})^{-1} \geq 1+\alpha$ almost everywhere on $A$, then for every open ball $B\subset D$ and for every $0<\beta\leq \frac{\alpha}{1+\alpha}$,
\[
B\subseteq A \quad \text{ if and only if } \quad \Lambda(1)+\beta \Lambda'(1)\chi_B \geq \E(\Lambda(\sigma)).
\]
\item If there exists $\alpha>0$ such that  $\E(\sigma_A) \leq 1-\alpha$ almost everywhere on $A$, then for every open ball $B\subset D$ and for every $0<\beta<\alpha$,
\[
B\subseteq A \quad \text{ if and only if } \quad \Lambda(1)-\beta \Lambda'(1)\chi_B\leq \E(\Lambda(\sigma)).
\]
\end{enumerate}
\end{theorem}
\begin{proof}
\begin{enumerate}[(a)]
\item For every $\alpha>0$ and $0<\beta\leq \frac{\alpha}{1+\alpha}$, it follows from \cite[Theorem 4.3]{harrach2013monotonicity} that
\[
\Lambda(1)+\beta \Lambda'(1)\chi_B\geq \Lambda(1+\alpha\chi_B).
\]
Hence, theorem \ref{thm:MonotonicityMethod} shows that
\[
B\subseteq A \quad \text{ implies } \quad \Lambda(1)+\beta \Lambda'(1)\chi_B \geq \E(\Lambda(\sigma)).
\]
The converse results from a similar combination of monotonicity and localized potentials arguments as above.
\item The claim follows as in part (a) by using \cite[Remark 4.5]{harrach2013monotonicity}.
\end{enumerate}
\end{proof}

\begin{remark}
In the next section we will numerically test the methods on the setting described in example \ref{ex:example1}(a)
where the considered anomalies consist of a finite number of subregions, or pixels whose (constant) conductivity levels are 
mutually independent uniformly distributed random variables. However, it should be noted that our theoretic results hold for general conductivity distributions, as long as $\sigma,\sigma^{-1}\in L^1(\Omega;L^\infty_+(D)) \subseteq L^1(\Omega;L^\infty(D))$ and \eqref{stoch_cond} is fulfilled. In particular, this includes log-normal distributed conductivities.
\end{remark}

\section{Numerical examples}
\label{Sec:numerics}
This section presents the numerical examples on the detection of stochastic inclusions from the mean of the Neumann-to-Dirichlet operator. We concentrate on the Factorization Method since it has thus far been more widely studied within the inverse problems community, making the comparison of our results to those achievable in the deterministic framework more straightforward. However, there is no reason to expect that the Monotonicity Method would perform any worse. 

\subsection{Simulation of data}
Let us first discuss how the expectation $\E(\Lambda(\sigma))$ of the Neumann-to-Dirichlet operator can be approximated for the numerical experiments that follow.
The idea is to compute the expected boundary voltages $\E(u_\sigma^g|_{\partial D}) = \E(u_\sigma^g)|_{\partial D}$ separately for a (finite) number of boundary currents $g$.

Once the distribution of the conductivity is given, the expectations $\E(u_\sigma^g)$ can be obtained by Monte Carlo simulation, i.e., by sampling from the distribution and solving a deterministic problem \eqref{eq:EIT} for each realization.
Another approach, which we employ here, is known as the \emph{stochastic finite element method} (SFEM).
We briefly sketch the main idea of SFEM under the assumption that $\sigma\in L^2(\Omega;L^\infty_+(D))$, which will be the case
in all our examples considered below. 
The forward problem describing EIT with a stochastic conductivity is to find $u_{\sigma(\omega)}^g\in \Hd^1(D)$ that solves
\begin{align}\labeq{stochEIT}
\nabla\cdot  \sigma(\omega) \nabla u_{\sigma(\omega)}^g  = 0 \quad {\rm in} \ D, \qquad \sigma(\omega) \partial_\nu u_{\sigma(\omega)}^g = g \quad {\rm on} \ \partial D
\end{align}
$P$-almost surely for a given current density $g \in \Ld^2(\partial D)$, which is equivalent to the standard variational formulation
\[
\int_D \sigma(\omega) \nabla u_{\sigma(\omega)}^g(x) \cdot \nabla \phi(x)  \dx \, 
  = \int_D g(x) \phi(x) \dx \qquad \text{ for all } \phi\in \Hd^1(D)
\]
$P$-almost surely. This is in turn equivalent to the extended variational formulation of finding $u_{\sigma}^g\in L^2(\Omega;\Hd^1(D))$ such that
\begin{equation}\labeq{ext_var}
\int_\Omega \int_D \sigma(\omega) \nabla u_{\sigma(\omega)}^g(x) \cdot \nabla \phi(x,\omega)  \dx  \dx[P] 
  = \int_\Omega  \int_D g(x) \phi(x,\omega) \dx \dx[P] 
\end{equation}
for all $\phi\in L^2(\Omega;\Hd^1(D))$.
The key idea of the stochastic finite element method is to approximate statistical properties, such as moments, of the solution to \req{stochEIT}
by applying the Galerkin method to the extended variational formulation \req{ext_var}.
Hence, we choose finite dimensional subspaces $V_\diamond^h\subset \Hd^1(D)$ and $S_m\subset L^2(\Omega)$
and determine the function $\tilde u_{\sigma}^g\in V_\diamond^h \otimes S_m\subset L^2(\Omega;\Hd^1(D))$ that solves \req{ext_var}
for all $\phi\in V_\diamond^h \otimes S_m\subset L^2(\Omega;\Hd^1(D))$.

As $V^h$ we use a standard piecewise linear finite element space. 
In all our examples below, the probability space $(\Omega,\Sigma,P)$ can be parametrized by a finite number of 
independent random variables 
$\theta_q:\Omega \rightarrow \Theta_q\subseteq\R$, $q=1,\ldots,Q$. Accordingly, we choose $S_m$ to be the space of all $Q$-variate polynomials with total degree less than or equal to $m$; the corresponding basis functions are chosen to be orthonormal,~i.e.,~suitably scaled multivariate Legendre polynomials in our setting.
Take note that the dimension of $S_m$ is
\[
M = \left(
\begin{array}{c}
\!\!\! Q + m \!\!\! \\
Q
\end{array}
\right)
= \frac{(Q + m)!}{Q! \, m!} \, .
\]
For more general conductivities one can consider, e.g.,\ truncated Karhunen--Lo\`{e}ve expansions of the random conductivity 
field in question.

Both the Factorization and the Monotonicity Method require only the difference operator $\E(\Lambda(\sigma)) - \Lambda(1)$ as the input.
Instead of approximating $u_{\sigma}^g$ as described above in order to calculate the expectation of $\Lambda(\sigma)$ and then subtracting the deterministic Neumann-to-Dirichlet map corresponding to the unit conductivity, a numerically more stable method is to directly compute the expectation $\E(\Lambda(\sigma) - \Lambda(1))$. This can be done by considering the boundary value problem satisfied by the difference potential (cf.,~e.g.,~\cite[Section~3]{Gebauer07})
  \begin{equation}\label{weak_w}
  \nabla\cdot  \sigma \nabla w_\sigma^g  = \nabla \cdot (\sigma - 1) \nabla u_0^g \quad {\rm in} \ D, \qquad \sigma \partial_\nu w_\sigma^g = 0 \quad {\rm on} \ \partial D ,
  \end{equation}
where $u_{0}^g:=u_{\sigma_0}^g$ is the solution to \eqref{eq:EIT} with background conductivity $\sigma_0 \equiv 1$.

Our numerical experiments are conducted in the unit disk $D \subset \R^2$, for which the most natural $L^2(\partial D)$-normalized boundary currents are the Fourier basis functions
  \begin{equation}
  \label{four_basis}
  g_{2t-1} = \frac{1}{\sqrt{\pi}} \sin(t \gamma), \quad g_{2t} = \frac{1}{\sqrt{\pi}} \cos(t \gamma), \qquad t=1,\ldots,T
  \end{equation}
for some $T \in \N$ and with $\gamma$ denoting the angular coordinate.
For these current densities, the gradient of the background solution $u_0^g$ can be presented in a closed form, which can be used explicitly on the right hand side of the variational formulation of \eqref{weak_w}. 
Using a stochastic finite element method as explained above, we solve \eqref{weak_w} for all current patterns in \eqref{four_basis}, 
and then expand the expectation value of the correspondingly approximated difference boundary data
$\E(\tilde{w}_\sigma^{g_t})$ in the basis \eqref{four_basis} by using FFT. This results in an approximate representation $L \in \R^{2T \times 2T}$ of $\E(\Lambda(\sigma)) - \Lambda(1): L^2_\diamond(\partial D) \to L^2_\diamond(\partial D)$ with respect to the Fourier basis~\eqref{four_basis}.

  \begin{figure}[t]
  \centering
  \includegraphics[scale=0.95]{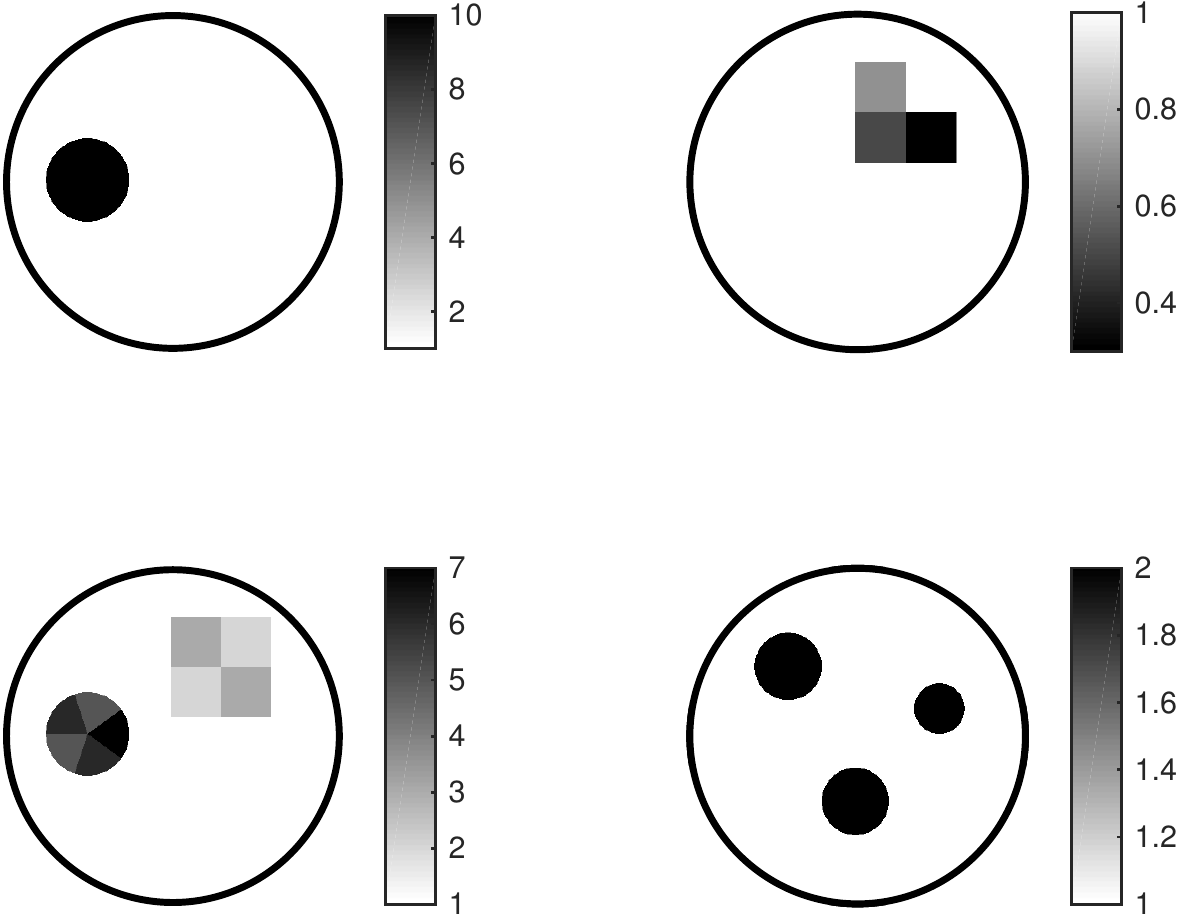}
  \caption{Deterministic inclusion geometries and mean conductivities. The random vector consists of 1,3,9 or 3 independent components that are uniformly distributed.}
  \label{fig:geom}
  \end{figure}

The considered inclusion geometries are illustrated in Figure~\ref{fig:geom}. In each case, the anomalies consist of a finite number of subregions, or pixels whose (constant) conductivity levels are mutually independent uniformly distributed random variables.
Thus, the random conductivity can be written as 
\begin{equation}
\label{num_sigma}
\sigma(\theta) = 1 + \sum_{q=1}^Q \theta_q \chi_q,
\end{equation}
where $\chi_q$ is the indicator function of the $q$th pixel.
The number of pixels $Q$ varies between one and nine depending on the configuration.

In all numerical tests we use third order stochastic basis polynomials, i.e., we choose $m=3$, and the deterministic FEM meshes are such that
the dimension of $V_h$ is approximately $33000$.
The resulting SFEM system matrices are very sparse \cite{Mustonen16}.
The number of employed Fourier basis functions is $2T = 100$ in all tests.

In addition to considering the unperturbed simulated measurement matrix $L \in \R^{2T \times 2T}$ that contains only numerical errors, we also test the Factorization Method after adding $0.1\%$ noise to $L$. More precisely, we generate a random matrix $E\in \R^{2T \times 2T}$ with uniformly distributed elements in the interval $[-1,1]$ and replace $L$ by its noisy version
\[
L_\epsilon := L + 10^{-3}\norm{L}_2\frac{E}{\norm{E}_2},
\]
where $\| \cdot \|_2$ denotes the spectral norm.
This is the same noise model as in \cite{Gebauer07} for the deterministic Factorization Method.

\subsection{Factorization Method}

Recall that Theorem~\ref{thm:factorization} provides a binary test for deciding whether a point is inside the anomaly $A$ or not:
Assuming that \eqref{stoch_cond} holds for some $\alpha > 0$, then for any $z \in D \setminus \partial A$, 
\[
\Phi_{z,d}|_{\partial D} \in \range \left( \left| \E(\Lambda(\sigma))- \Lambda(1) \right|^{1/2} \right)  \quad \Longleftrightarrow 
\quad z \in D.
\]
To numerically implement this test, we follow~\cite{Gebauer07}. The boundary potential $\Phi_{z,d}|_{\partial D}$ belongs to $\range ( | \E(\Lambda(\sigma))- \Lambda(1)|^{1/2} )$ if and only if it satisfies the so-called {\em Picard criterion}, which compares the decay of the scalar projections of $\Phi_{z,d}|_{\partial D}$ on the eigenfunctions of the compact and self-adjoint operator $|\E(\Lambda(\sigma))- \Lambda(1) |^{1/2}$ to the decay of the corresponding eigenvalues. The leading idea is to write an approximate version of the Picard criterion by using the singular values and vectors of $L \in \R^{2T \times 2T}$ in place of the eigensystem of $|\E(\Lambda(\sigma))- \Lambda(1) |$. We refer to \cite{Lechleiter06} for a proof that this approach can be formulated as a regularization strategy.

Let us introduce a singular value decomposition for the discretized boundary operator $L\in \R^{2T \times 2T}$, that is,
\[
L v_t = \lambda_t u_t, \qquad L^{\rm T} u_t = \lambda_t v_t, \qquad t=1,\ldots, 2T,
\]
with nonnegative singular values $\{ \lambda_t\} \subset \R$, which are sorted in decreasing order, and orthonormal bases $\{ u_t\}, \{ v_t\}\subset \R^{2 T}$.
Denoting by $F_{z,d} \in \R^{2T}$ the Fourier coefficients of $\Phi_{z,d}$,
we introduce an indicator function
\[
{\rm Ind}(y; d,\tau) :=
\sum_{t=1}^\tau (F_{z,d} \cdot v_t)^2 \Big/ \sum_{t=1}^\tau \frac{(F_{z,d} \cdot  v_t)^2}{\lambda_t},
\]
where $1 \leq \tau \leq 2T$ should be chosen so that $\lambda_{\tau+1}$ is the first singular value below the expected measurement/numerical error. Note that the formation of $F_{z,d}$ is trivial because the functional form of $\Phi_{z,d}$ is known explicitly when $D$ is the unit disk (cf.,~e.g.,~\cite{Bruhl00}). It is to be expected that ${\rm Ind}: D \to \R_+$ takes larger values inside than outside the anomaly $A$ (cf.,~e.g.,~\cite{Gebauer07}), meaning that plotting ${\rm Ind}$ should provide information on the whereabouts of $A$. In our numerical tests, we plot the indicator function on an equidistant grid that is chosen independently of the finite element meshes used for solving the forward problems. The employed dipole moments are $d_r = (\cos(2 \pi r /3), \sin(2 \pi r /3)) \in \R^2$ for $r \in \{1,2,3\}$ and the indicator function used in reconstructions is actually the sum of the indicator functions ${\rm Ind}(y;d_r,\tau)$ corresponding to those three dipole moments.
For all noiseless experiments, we use the maximal spectral cut-off index $\tau = 2T = 100$. (We remind the reader that the Factorization Method is not designed to produce information about the conductivity of the inhomogeneity. In other words, only the `support' of the indicator function is relevant.)

\begin{figure}[t]
  \centering
  \includegraphics[scale=0.95]{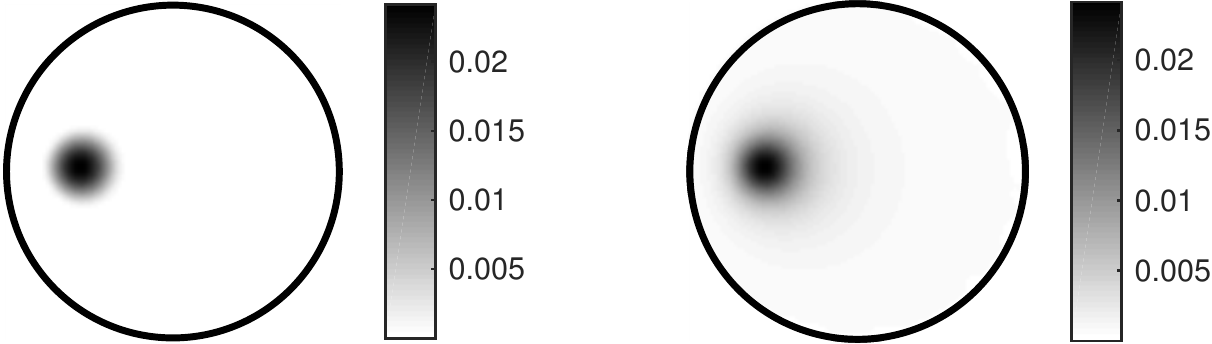}
  \caption{Reconstructions corresponding to the top left target conductivity in Figure~\ref{fig:geom}. Left: noiseless data. Right: noisy data. (See text for details.)}
  \label{fig:num1}
  \end{figure}

{\bf Test 1.} Our first and simplest test case considers one disk-shaped inclusion, as depicted in the top left image of Figure~\ref{fig:geom}.
We choose $\Theta_1 = [8,10] \subset \R$ so that all realizations of the stochastic inclusion significantly differ from the unit background conductivity. In this case, $\E(\sigma_A^{-1})^{-1} \approx 9.97$, which obviously satisfies part (a) of \eqref{stoch_cond}.
Figure~\ref{fig:num1} shows the indicator function for both noiseless and noisy ($\tau = 6$) Neumann-to-Dirichlet data; according to a visual inspection, the support of the indicator function matches well with the inclusion in both cases. This is not a big surprise: All possible realizations of the random conductivity satisfy the requirements of the {\em deterministic} Factorization Method, and thus one would expect reconstructions that are comparable to those in the deterministic setting (cf.,~e.g.,~\cite{Gebauer07}).

 \begin{figure}[t]
  \centering
  \includegraphics[scale=0.95]{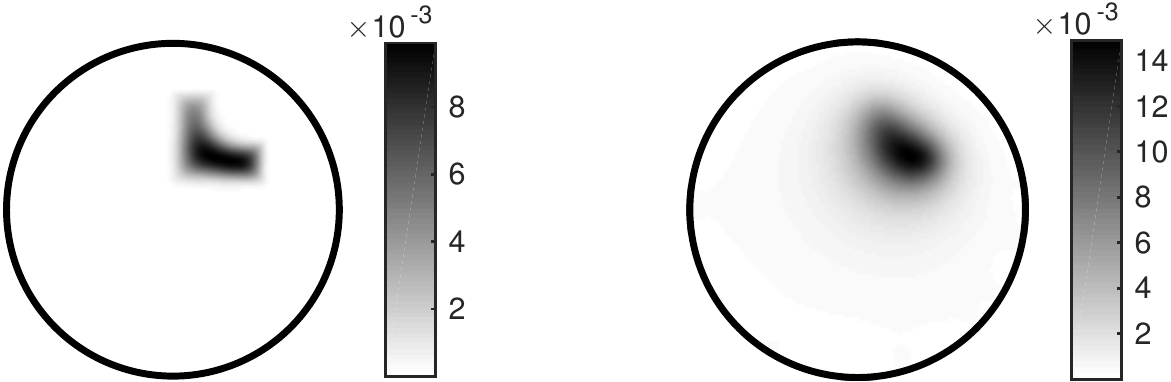}
  \caption{Reconstructions corresponding to the top right target conductivity in Figure~\ref{fig:geom}. Left: noiseless data. Right: noisy data. (See text for details.)}
  \label{fig:num2}
  \end{figure}

{\bf Test 2}. The second experiment considers the nonconvex inclusion 
in the top right image of Figure~\ref{fig:geom}.
The expected conductivity values of the three pixels constituting the anomaly are now strictly less than the unit background.
More precisely, the pixelwise values vary according to $\theta_1 \in [-0.99, -0.41]$, $\theta_2 \in [-0.99, -0.01]$ and $\theta_3 \in [-0.99, 0.39]$; see \eqref{num_sigma}. Note that the condition (b) of~\eqref{stoch_cond} is satisfied, even though one of the three pixels can have either positive or negative contrast. The reconstructions produced by the Factorization Method are presented in Figure~\ref{fig:num2}. Even without noise, the nonconvex edge of the inclusion is clearly smoothened, but otherwise the noiseless reconstruction is accurate. In the noisy case, the indicator function ($\tau = 7$) contains practically no information on the nonconvexity of the inclusion, but the size and the location of the anomaly are still reproduced accurately. 

\begin{figure}[t]
  \centering
  \includegraphics[scale=0.95]{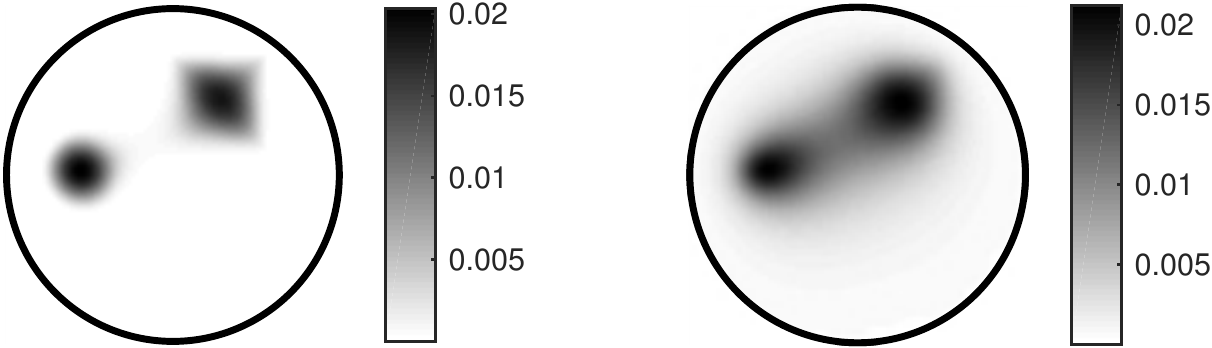}
  \caption{Reconstructions corresponding to the bottom left target conductivity in Figure~\ref{fig:geom}. Left: noiseless data. Right: noisy data. (See text for details.)}
  \label{fig:num3}
  \end{figure}

{\bf Test 3.} In the third example, we have two separate inclusions and altogether nine random variables as shown in the bottom left image of Figure~\ref{fig:geom}. The conductivity intervals have width of six units for the pixels in the disk and width of three units for those in the square. All realizations of the discoidal inclusion are thus more conductive than the background, while the conductivity contrasts of some pixels in the square anomaly may be either positive or negative. The condition (a) of~\eqref{stoch_cond} is anyway satisfied: $\E(\sigma_A^{-1})^{-1} > 4.3$ for the disk, and $\E(\sigma_A^{-1})^{-1} > 1.5$ for the square. The reconstruction from noiseless data in Figure~\ref{fig:num3} accurately indicates the shapes and sizes of the two inclusions. 
With noisy data, the reconstruction ($\tau = 12$) tends to produce one large inclusion instead of two smaller ones. The Factorization Method has been observed to exhibit similar bridging behavior also in the deterministic setting (see,~e.g.,~\cite{Gebauer07}).

 \begin{figure}[t]
  \centering
  \includegraphics[scale=0.95]{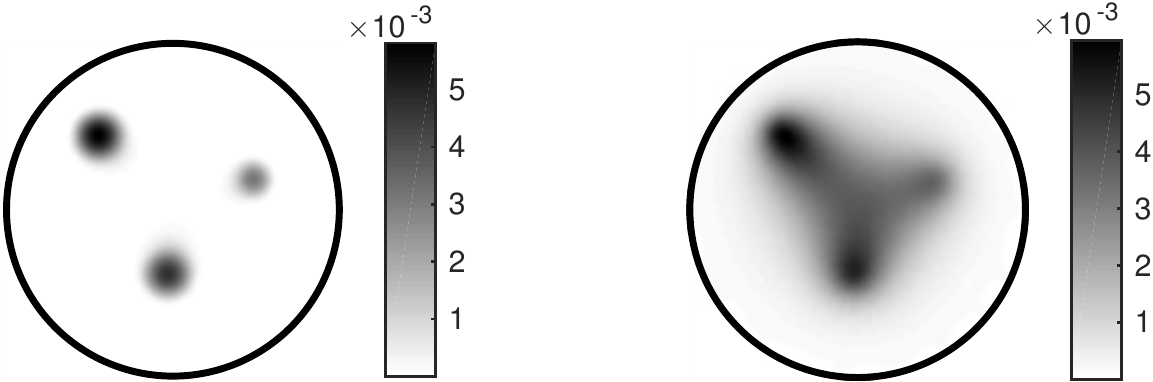}
  \caption{Reconstructions corresponding to the bottom right target conductivity in Figure~\ref{fig:geom}. Left: noiseless data. Right: noisy data. (See text for details.)}
  \label{fig:num4}
  \end{figure}

{\bf Test 4.} Finally, we test the Factorization Method for three well separated disk-like inclusions, each characterized by a uniformly distributed conductivity in the interval $[0.5,3.5]$; see  the bottom right image of Figure~\ref{fig:geom}.
Note that the condition (a) of~\eqref{stoch_cond} is satisfied as $\E(\sigma_A^{-1})^{-1} \approx 1.54$, although each realization of the anomaly can have positive, negative or indefinite contrast compared to the unit background.
Once again, the noiseless reconstruction in Figure~\ref{fig:num4} is very accurate, but measurement noise deteriorates the representation of the inclusion boundaries in the plot of the indicator function ($\tau = 11$).

\section{Concluding remarks}
We have shown that an application of the Factorization Method or the Monotonicity Method to the mean value of the Neumann-to-Dirichlet boundary map characterizes stochastic anomalies, i.e., inclusions with random conductivities but fixed shapes, assuming that the contrast of the anomalies is high enough in the sense of expectation. This theoretical result was complemented by two-dimensional numerical examples demonstrating the functionality of the Factorization Method in this stochastic setting.

\bibliography{literaturliste}
\bibliographystyle{abbrv}

\end{document}